\documentclass[a4paper,reqno]{article}
\usepackage{amsmath,amsthm,amssymb,amsfonts,euscript,graphicx}
\usepackage{multicol,multirow}
\usepackage{pst-node,pst-plot,tikz}
\usetikzlibrary{shapes,snakes}
\usepackage{mathdots, comment}
\pstVerb{realtime srand}
\psset{plotpoints=50}
\usepackage{qtree}
\usepackage{algorithm}
\usepackage[noend]{algpseudocode}
\makeatletter
\def\BState{\State\hskip-\ALG@thistlm}
\makeatother

\theoremstyle{plain}
\newtheorem{theorem}{Theorem}
\newtheorem{lemma}{Lemma}
\newtheorem{corollary}{Corollary}
\newtheorem{proposition}{Proposition}

\theoremstyle{definition}

\newtheorem{remark}{Remark}
\newtheorem{example}{Example}

\begin{document}

\title{\large \bf On linear version of an elementary group theory result
}

\author{{\normalsize{\sc Mohsen Aliabadi}}\\
{\footnotesize{\it Department of Mathematics, Statistics, and Computer Science, University of Illinois at Chicago,   }}\\
{\footnotesize{\it 851 S. Morgan St, Chicago, IL 60607, USA}}\\
{\footnotesize{E-mail address: $\mathsf{maliab2@uic.edu}$}}\\
}
\date{}

\maketitle
\begin{abstract}
Given a cyclic group $G$ of order $p^r$, where $p$ is a prime and $r\in\mathbb{N}$. It is  well-known that the order of its greatest proper subgroup
$\psi(G)$ and the number of   its generating elements $\phi(G)$ satisfy  $\psi(G)+\phi(G)=p^r$. In this paper, we study a linear version of this group theory theorem for primitive subspaces in a field extension using some linear algebra tools. We also briefly discuss partitions of finite fields by using their primitive subspaces. 

\textbf{Notation:} The cardinality of a set $S$ shall  be denoted by $\#S$.   For a vector space $V$ over a field $K$, we denote its dimension by $\dim_KV$. The number of positive divisors of a positive integer $n$ is denoted by $d(n)$. For simplicity of notation,  $A\subset B$ means that $A$ is a proper subset of $B$. 
\end{abstract}

\section{Introduction}
There has been a vast literature as well as ongoing investigations on linear
analogues of existing results in group theory. As a case  in point, a recent result  due to Bachoc et al \cite{2}, which gives the linear version of a theorem of Kneser on the size of   certain subsets of an abelian group. One can also see \cite{1} in which the linear analogue of a group theory result is studied. In this paper, we start by considering the following scenario in group theory.

Let $\mathbb{Z}/p^r\mathbb{Z}$ denote the cyclic group of order $p^r$, where $p$ is a prime and $r\in \mathbb{N}$. Denote by $\psi(\mathbb{Z}/p^r\mathbb{Z})$ the order of greatest proper subgroup of $\mathbb{Z}/p^r\mathbb{Z}$ and denote by $\phi(\mathbb{Z}/p^r\mathbb{Z})$ the number of generators of $\mathbb{Z}/p^r\mathbb{Z}$. Since $p$-groups have subgroups of index $p$, then $\psi(\mathbb{Z}/p^r\mathbb{Z})=p^{r-1}$. Also, it is well-known that $\phi(\mathbb{Z}/p^r\mathbb{Z})=\varphi(p^r)$, where $\varphi$ stands for the Euler's totient function. According to the Euler's product formula, $\varphi(p^r)=p^r\left(1-\frac{1}{p}\right)=p^r-p^{r-1}$. Therefore,
\begin{align} \label{f1}
\psi(\mathbb{Z}/p^r\mathbb{Z})+\phi(\mathbb{Z}/p^r\mathbb{Z})=p^r.
\end{align}

Let $K\subset F$ be a finite separable field extension with $[F:K]=n$.  Let $A$ be a $K$-subspace of $F$. We say that   $A$ is a \textit{primitive} $K$-subspace of $F$ if $K(a)=F$, $\forall a\in A\setminus\{0\}$. (Note that since $K\subset L$ is a finite separable extension, then it is simple, which justifies the existence of primitive subspaces.)

\begin{example}
Consider the field extension $\mathbb{C}/\mathbb{R}$ and the $\mathbb{R}$-subspace $W=\langle i\rangle$ of $\mathbb{C}$, where $\langle i\rangle$ stands for the $\mathbb{R}$-subspace of $\mathbb{C}$ generated by $i$. Then adjoining any non-zero element of $W$ to $\mathbb{R}$ covers the whole $\mathbb{C}$. So $W$ is a primitive $\mathbb{R}$-subspace of $\mathbb{C}$.
\end{example}
Motivated by the aforementioned theorem on the number of generating elements of $\mathbb{Z}/p^r\mathbb{Z}$, one naturally inquires about the size of a primitive subspace $A$. In
\cite{1.2} it is proven that $\dim_KA\leq n-\psi(F,K)$, where 
\[\psi(F,K)=\max\left\{[M:K]:\, M\text{ is a proper intermediate field of }K\subset F\right\}.\]
(Note that in the above definition, ``proper intermediate field of $K\subset F$"   means $K\subseteq M\subset F$. Hence we have $1\leq\psi(F,K)<[F:K]$.)

More generally, the following results in \cite{1.2} provides the linear version of (\ref{f1}).

\begin{proposition}\label{th2.1}
Let $F, K, n$ and $\psi(F,K)$ be as above. Assume that $K$ is infinite and  $K\subset F$ is simple.  Then 
\begin{align}\label{eq1}
\psi(F,K)+\phi(F,K)=n,
\end{align}
where
\[\phi(F,K)=\max\left\{\dim_KV:\, V\text{ is a primitive }K-\text{subspace of }F\right\},\]
namely, $\phi(F,K)$ denotes the dimension of the largest primitive subspace.
\end{proposition}

\begin{example}\label{er}
Consider the finite field extension $\mathbb{Q} (\sqrt{2}, \sqrt{3})/\mathbb{Q}$. Then according to Proposition \ref{th2.1} the dimension of greatest primitive $\mathbb{Q}$-subspace of $\mathbb{Q} (\sqrt{2}, \sqrt{3})$ is $2$ as $[\mathbb{Q} (\sqrt{2}, \sqrt{3}):\mathbb{Q}]=4$ and so $\psi (\mathbb{Q} (\sqrt{2}, \sqrt{3}),\mathbb{Q})=2$.
\end{example}

\begin{remark}
From now on, we call Proposition \ref{th2.1} as ``the primitive subspace theorem". Note that the linear analogues of ``the  order  of group", ``the order of its largest proper subgroup" and ``the number of generators of a cyclic group" are ``the degree of field extension, ``the degree of its largest proper intermediate subfield" and ``the dimension of largest primitive vector space", respectively.
\end{remark}
In this paper, we will generalize the primitive subspace theorem. We will also give a positive answer to the question inquired as Remark $5$ in \cite{1.2} whether (\ref{eq1}) holds in case the base field is finite, for some special cases.

In the proof of our main result, we utilize a linear algebra theorem, which asserts that a
vector space over an infinite field cannot be written as a finite union of its
proper subspaces. 

\section{Preliminaries}
In what follows, we  provide a   proof for the well-known linear algebra result that ``finite union of lower-dimensional subfields of number fields (considered as vector spaces over $\mathbb{Q}$) do not cover the field". Our proof, to the best of our knowledge, has not been provided before in literature. One can see \cite[Chapter 8]{2n} and \cite[1.10.3-6]{*} for existing proofs. 
\begin{lemma}\label{l1}
Let $V$ be a vector space over an infinite field $K$. Then $V$ is not a union of a finite collection of proper subspaces.
\end{lemma}
\begin{proof}
Suppose $V$ is the union of $m$ proper subspaces $U_i$ for $1\leq i\leq m$, where $m$ is as small as possible, and thus $m\geq2$. By the minimality of $m$, each subspace $U_i$ contains a vector $a_i$, such that $a_i$ is not contained in $U_j$ for any subscript $j\neq i$. (Otherwise, we could delete the subspace $U_i$ from the union.) Fix the vectors $a_i$ for $1\leq i\leq m$.

For $t\in K$, write 
\[f(t)=\sum_{i=1}^m a_i t^i.\]
Since $K$ is infinite, we can choose $m$ distinct elements $t_j\in K$ such that the vectors $b_j=f(t_j)$ all lie in the same subspace, say $U_1$. We thus have 
\[b_j=f(t_j)=\sum_{i=1}^m a_i(t_j)^i\in U_1\]
for $1\leq j\leq m$.

We thus have the matrix equation $\mathbf{aM=b}$, where 
\[\mathbf{a}=(a_1,a_2,\ldots,a_m)\quad\mathrm{and}\quad\mathbf{b}=(b_1,b_2,\ldots,b_m),\]
and $\mathbf{M}=[(t_j)^i]$.

The Vandermonde matrix $\mathbf{M}$ is invertible because the $t_j$ are distinct, so we have $\mathbf{a=bM}^{-1}$, and thus each component $a_i$ of $\mathbf{a}$ is a $K$-linear combination of the vectors $b_j$. Since $b_j\in U_1$ for all $j$, it follows that $a_i\in U_1$ for all $i$. But $a_2\notin U_1$, so this is a contradiction.
\end{proof}

\begin{remark}
It is worth pointing out that Lemma \ref{l1} does not hold if the base field $K$ is finite. To see this, just consider the vector space $\mathbb{F}_{2}^{2}$ over $\mathbb{F}_{2}$, which is the union of its three proper subspaces.
\end{remark}

A key ingredient of the proof of the analogue of the primitive subspace theorem for finite fields is a theorem on covering a given vector space over a finite field by its subspaces.

Let $V$ be a finite-dimensional vector space over $\mathbb{F}_{q}$, where $\mathbb{F}_{q}$ stands for finite field of order $q$, which  $q=p^r$ for some prime $p$ and  $r\in\mathbb{N}$. We say  a collection $\{W_i\}_{i\in I}$ of proper $K$-subspaces of $V$ is  a \textit{linear covering} of $V$ if $V=\underset{i\in I}{\bigcup} W_i$. 
 The \textit{linear covering number} $\mathrm{LC}(V)$ of a vector space $V$ of dimension at least $2$ is the least cardinality $\#$I of a linear covering $\{W_{i}\}_{i \in I}$ of $V$. Under the condition $\dim_KV\geq2$, which is the sufficient and necessary condition for the existence of linear coverings, we have the following result from \cite{(a.1)}.
\begin{proposition}\label{thB}
If $\dim_KV$ and $\#K$ are not  both infinite, then  $\mathrm{LC}(V)=\#K+1$.
\end{proposition}

Taking advantage of Proposition \ref{thB} and implementing a similar technique used in the proof of Theorem 7 in \cite{1.2} (mentioned as Proposition \ref{th2.1} in this paper) one can prove the analogue of Proposition \ref{thB} for finite base field in some special cases depending on the degree of extension. The general case of the primitive subspace theorem for finite fields is left incomplete.
\\
In our proof, Proposition \ref{thB} serves as a substitute for Lemma \ref{l1} used for infinite field case. We will also use a classical result in field theory which states that for any intermediate subfield $L$ of $\mathbb{F}_{q} \subset \mathbb{F}_{q^{n}}$, we have $\# L= q^{d}$, where $d\mid n$.
\begin{theorem}\label{th2.2}
Consider the finite field extension $\mathbb{F}_{q}\subset \mathbb{F}_{q^{n}}$, where $n\in\mathbb{N}$, and $q=p^r$ for some prime $p$ and $r\in\mathbb{N}$. Let $d(n)<q+2$, where $d(n)$ denotes the number of positive divisors of $n$. Then 
\begin{align}\label{eq2}
\psi(\mathbb{F}_{q^{n}},\mathbb{F}_{q})+\phi(\mathbb{F}_{q^{n}},\mathbb{F}_{q})=n.
\end{align}
Here $\psi$ and $\phi$ are defined in the same way  as in  Theorem \ref{th2.1}.
\end{theorem}
\begin{proof}
Assume that $\{M_i\}_{i=1}^m$ is the family of all proper intermediate subfields of $\mathbb{F}_{q}\subset \mathbb{F}_{q^{n}}$. Clearly $m+1=d(n)$. Define $\mathcal{A}$ as the set of all primitive $\mathbb{F}_q$-subspaces of $\mathbb{F}_{q^{n}}$.
(Note that $\mathcal{A}\neq\varnothing$ because $\mathbb{F}_{q^n}^*$ is a cyclic group and adjoining any generator $a$ of $\mathbb{F}_{q^n}^*$ to $\mathbb{F}_q$ will cover $\mathbb{F}_{q^n}$, so $\langle a\rangle\in\mathcal{A}$, where $\langle a\rangle$ stands for the $\mathbb{F}_q$-subspace generated by $a$.)
Without loss of generality assume that $\psi(\mathbb{F}_{q^{n}},\mathbb{F}_{q})=[M_1:\mathbb{F}_{q}]$. Choose $W\in\mathcal{A}$ for which $\dim_{\mathbb{F}_{q}}W=\phi(\mathbb{F}_{q^{n}},\mathbb{F}_{q})$. According to Proposition \ref{thB} and the fact that $d(n)<q+2$, we have $\mathbb{F}_{q^{n}}\neq\overset{m}{\underset{i=1}{\bigcup}}M_i$. Choose $x_1\in\mathbb{F}_{q^{n}}\setminus\overset{m}{\underset{i=1}{\bigcup}}M_i$ and define $M_{i,1}=M_i\oplus\langle x_1\rangle$, for $1\leq i\leq m$. So $\{M_{i,1}\}_{i=1}^m$ is a finite family of proper $\mathbb{F}_{q}$-subspaces of $\mathbb{F}_{q^{n}}$. Likewise, $\mathbb{F}_{q^{n}}\neq\overset{m}{\underset{i=1}{\bigcup}}M_{i,1}$. Choose $x_2\in F\setminus \overset{m}{\underset{i=1}{\bigcup}}M_{i,1}$ and define $M_{i,2}=M_{i,1}\oplus\langle x_2\rangle$, for $1\leq i\leq m$. Similarly, we get $\mathbb{F}_{q^{n}}\neq\overset{m}{\underset{i=1}{\bigcup}}M_{i,2}$. Continuing in this manner, we obtain a family of $\mathbb{F}_{q}$-subspaces $M_{i,j}$ of $\mathbb{F}_{q^{n}}$, $1\leq i\leq m$ and $1\leq j\leq n-\psi(\mathbb{F}_{q^{n}},\mathbb{F}_{q})$. Consider the $\mathbb{F}_{q}$-subspace $V$ of $\mathbb{F}_{q^{n}}$ spanned by $\{x_1,x_2,\ldots,x_{n-\psi(\mathbb{F}_{q^{n}},\mathbb{F}_{q})}\}$. For  any $v\in V\setminus\{0\}$, we have $v\notin \overset{m}{\underset{i=1}{\bigcup}}M_i$. It follows that $\mathbb{F}_{q}(v)=F$. Therefore $V\in\mathcal{A}$. This implies $\phi(\mathbb{F}_{q^{n}},\mathbb{F}_{q})\geq\dim_{\mathbb{F}_{q}}V=n-\psi(\mathbb{F}_{q^{n}},\mathbb{F}_{q})$. We claim that $\phi(\mathbb{F}_{q^{n}},\mathbb{F}_{q})=n-\psi(\mathbb{F}_{q^{n}},\mathbb{F}_{q})$ because otherwise $\dim_{\mathbb{F}_{q}}W>\dim_{\mathbb{F}_{q}}V$ which yields $[M_1:\mathbb{F}_{q}]+\dim_{\mathbb{F}_{q}}W>n$. This follows $M_1\cap W\neq\{0\}$ which is a contradiction (Note that if $x\in(M_1\cap W)$ is nonzero then $x$ cannot be a primitive element of $\mathbb{F}_{q}\subset \mathbb{F}_{q^{n}}$.) Therefore $\dim_{\mathbb{F}_{q}}V=\dim_{\mathbb{F}_{q}}W$. Our proof is complete.
\end{proof}
\begin{example}
Consider the finite field extension $\mathbb{F}_{13^8}/\mathbb{F}_{13^2}$. Then according to Theorem \ref{th2.2} the dimension of greatest primitive $\mathbb{F}_{13^2}$-subspace of $\mathbb{F}_{13^8}$ is 2 as $\left[\mathbb{F}_{13^8}:\mathbb{F}_{13^2}\right]=4$ and $\psi\left(\mathbb{F}_{13^8}:\mathbb{F}_{13^2}\right)=2$.
\end{example}

\begin{remark}
The linear analogue of the group theory theorem presented
as Proposition \ref{th2.1} for infinite base field (or Theorem \ref{th2.2} for finite fields) seems a better result than the original group theory theorem
which only applies to cyclic groups of ``prime power" order, as it applies to all finite degree extensions $F/K$ that are 	``cyclic" (i.e., monogenic as a $K$-algebra). Formulating a reasonable analogue for all finite cyclic groups shall be possible in some ways. 
\end{remark}

\section{Main Result}
In this section we generalize the primitive subspace theorem in case the base field $K$ is infinite. We  begin with the following Lemma:
\begin{lemma}\label{l2}
Let $A$, $B$ and $C$ be subspaces of a vector space $V$, and suppose $A\cap B=0$ and $(A+B)\cap C=0$. Then $(A+C)\cap B=0$.
\end{lemma}
\begin{proof}
If $(A+C)\cap B\neq0$, let $b\in (A+C)\cap B$, where $b\neq0$, and write $b=a+c$, with $a\in A$ and $c\in C$.

Then $c=b-a$ lies in $C\cap (A+B)=0$, so $c=0$, and thus $a=b$ lies in $A\cap B=0$, and thus $a=0=b$. This contradicts the fact that $b\neq0$.
\end{proof}
\begin{theorem}\label{th4}
Let $K$ be an infinite field and $\mathcal{F}$ be a nonempty finite collection of proper subspaces of a $K$-vector space $V$, where $\dim_K(V)=n<\infty$. Let $s=\max\{\dim_K(S)\mid S\in\mathcal{F}\}$. Let $T\subseteq V$ be a subspace maximal with the property that $T\cap S=0$ for all $S\in\mathcal{F}$. Then $\dim_K(T)=n-s$.
\end{theorem}
\begin{proof}
Let $t=\dim_K(T)$. If $S\in\mathcal{F}$, then $T\cap S=0$, so 
\[n=\dim_KV\geq \dim_K(T+S)=\dim_K(T)+\dim_K(S)=t+\dim_K(S),\]
and thus $\dim_K(S)\leq n-t$ for all $S\in\mathcal{F}$. Thus $s\leq n-t$, and so $t\leq n-s$.

To complete the proof, we show that $t=n-s$. Otherwise, $t<n-s$, so $n>s+t$. Then
\[\dim_K(V)=n>s+t\geq\dim_K(S)+\dim_K(T)\geq\dim_K(S+T),\]
for all $S\in\mathcal{F}$. Then $S+T$ is a proper subspace of $V$ for all $S\in\mathcal{F}$. 

By Lemma \ref{l1}, there exists a vector $v\in V$ such that $v\notin S+T$, and thus $(S+T)\cap Kv=0$ for all $S\in\mathcal{F}$. Also, by the definition of $T$, we have $S\cap T=0$, and thus by Lemma \ref{l2}, we have $(Kv+T)\cap S=0$ for all $S\in\mathcal{F}$. Now $v\notin S+T$, so $v\notin T$, and thus $T \subsetneq Kv+T$. This contradicts the definition of $T$.
\end{proof}
\begin{corollary}
The primitive subspace theorem is a special case of Theorem \ref{th4}.
\end{corollary}
\begin{proof}
Let $K,F,n,\psi(F,K)$ and $\phi(F,K)$ be as Proposition \ref{th2.1}. Let $K$ be infinite, $K\subset F$ be simple, $\mathcal{F}$ be the collection of all proper intermediate subfields of $K \subset F$ and $s=\max\{[S:K]:S\in\mathcal{F}\}$. Assume that $T\subset F$ be a subspace maximal with the property that $T\cap S=0$, for all $S\in \mathcal{F}$. Then according to 
Theorem \ref{th4}, $\dim_K(T)=n-s$. On the other hand, we know that $\dim_K(T)=\phi(F,K)$ and $s=\psi(F,K)$. This implies the primitive subspace theorem for infinite fields, as claimed.
\end{proof}

\begin{remark}
The generalized primitive subspace theorem for infinite fields shall be reformulated for finite fields in some ways. Further investigation along those lines could prove to be worthwhile.
\end{remark}

\section{Partitioning Finite Fields}
Consider the field extension $\mathbb{F}_q\subset\mathbb{F}_{q^n}$, where $n\in \mathbb{N}$ and $q=p^{r}$ for some prime $p$ and $r \in \mathbb{N}$. Let $V$ be an $\mathbb{F}_q$-subspace of $\mathbb{F}_{q^n}$. We call a set $\mathcal{P}=\{W_i\}_{i=1}^\ell$ of $\mathbb{F}_q$-subspaces of $\mathbb{F}_{q^n}$ a \textit{partition} of $V$ if every non-zero element of $V$ is in $W_i$ for exactly one $i$. See \cite{(a)} for more results on partitions of finite vector spaces.

In the following theorem, we provide a partition for $\mathbb{F}_{q^n}$ using its primitive $\mathbb{F}_q$-subspaces.
\begin{theorem}
Let $M_1$, $W$ and $\psi(\mathbb{F}_{q^n},\mathbb{F}_q)$ be as in Theorem \ref{th2.2}. Assume that $W$ has a subspace partition $\{W_1,\ldots,W_l\}$, where $\dim_{\mathbb{F}_q}W_i=t_i\leq\psi(\mathbb{F}_{q^n},\mathbb{F}_q)$, for $1\leq i\leq l$. Then, for each $\alpha\in M_1\setminus\{0\}$, one can define a $t_i$-dimensional subspace $W_{i_\alpha}$ of $\mathbb{F}_{q^n}$ such that $W$, $M_1$ and the subspaces $W_{i_\alpha}$ form a partition of $\mathbb{F}_{q^n}$.
\end{theorem}
\begin{proof}
For each subspace $W_i$, $1\leq i\leq l$, let $T_i$ be a 1-1 linear transformation frow $W_i$ into $M_1$. For each $\alpha\in M_1\setminus\{0\}$, we associate with it  the following set:
\begin{align*}
W_{i_\alpha}=\{w+\alpha T_i(w):\ w\in W_i\}.
\end{align*}
Since $W_i\cap M_1=\{0\}$ and $W_i\cap W_j=\{0\}$, for $i\neq j$, one can easily verify that $W_{i_\alpha}$'s, $M_1$ and $W$ form a partition of $\mathbb{F}_{q^n}$ into  subspaces.
\end{proof}

\section*{Acknowledgment}
I am deeply grateful to Prof. Shmuel Friedland for his constant encouragement and for many insightful conversations.

\end{document}